\documentclass[11pt,a4paper]{amsart}
\usepackage{graphics,epic}
\usepackage{amsmath,amssymb, amsthm}
\usepackage[all,2cell,color]{xy}
\usepackage{xcolor}

\newtheorem{theorem}{Theorem}[section]
\newtheorem*{theorem*}{Theorem}

\newtheorem{lemma}[theorem]{Lemma}
\newtheorem{proposition}[theorem]{Proposition}
\newtheorem{corollary}[theorem]{Corollary}

\newtheorem*{conjecture*}{Conjecture}
\theoremstyle{remark}

\newtheorem{example}[theorem]{Example}
\newtheorem{remark}[theorem]{Remark}

\newcommand{\ie}{{\em i.e.}\ }

\newcommand{\opname}[1]{\operatorname{\mathsf{#1}}}

\newcommand{\add}{\opname{add}\nolimits}

\renewcommand{\Im}{\opname{Im}\nolimits}

\newcommand{\rank}{\mathrm{rank}}

\newcommand{\rad}{\opname{rad}}

\newcommand{\ra}{\rightarrow}

\newcommand{\id}{\mathrm{id}}

\newcommand{\ten}{\otimes}
\newcommand{\lten}{\overset{\boldmath{L}}{\ten}}
\newcommand{\rhom}{\mathbf{R}\mathrm{Hom}}

\newcommand{\cd}{{\mathcal D}}

\newcommand{\ci}{{\mathcal I}}

\newcommand{\cs}{{\mathcal S}}

\numberwithin{equation}{section}

\begin{document}

\title[Failure of DJHP]{Examples of finite dimensional algebras which do not satisfy the derived Jordan--H\"older property}
\author{Qunhua Liu}
\address{Qunhua Liu, Institute of Mathematics, School of Mathematical Sciences, Nanjing Normal University, Nanjing 210023, P.R.China} \email{05402@njnu.edu.cn}
\author{Dong Yang}
\thanks{Correspondence to be sent to Dong Yang yangdong@nju.edu.cn.}
\address{Dong Yang, Department of Mathematics, Nanjing University, Nanjing 210093, P. R. China}
\email{yangdong@nju.edu.cn}

\date{Last modified on \today}

\begin{abstract} We construct a matrix algebra $\Lambda(A,B)$ from two given finite dimensional elementary algebras $A$ and $B$ and give some sufficient conditions on $A$ and $B$ under which the derived Jordan--H\"older property (DJHP) fails for $\Lambda(A,B)$. This provides finite dimensional algebras of finite global dimension which do not satisfy DJHP.
\end{abstract}

\maketitle


\section{Introduction}

Recollements of triangulated categories, defined in \cite{BBD}, provide a tool of decomposing triangulated categories into smaller ones. They are analogues of short exact sequences of groups or modules.
In this paper we restrict our attention to recollements whose terms are all derived categories of algebras.
Derived simple algebras, analogous to simple groups or simple modules, are those algebras whose derived categories cannot be decomposed further via recollements. For example, local algebras are derived simple \cite{AKL2}. In analogy to composition series, a stratification of a derived category is a sequence of recollements which iteratively decomposes the given category into derived simple ones. It is natural to ask for a derived version of Jordan--H\"older property (DJHP for short), that is, the existence and uniqueness of finite stratifications of derived categories.  It was shown that hereditary Artin algebras \cite{AKL2}, finite dimensional piecewise hereditary algebras \cite{AKL3}, commutative Noetherian rings \cite{AKLY2} and derived discrete algebras \cite{Q} satisfy DJHP.

DJHP fails for general algebras by results of Chen and Xi \cite{CX12, CX15}. The first finite dimensional example was proposed also by Chen and Xi. Another finite dimensional example was given in \cite[Example 7.6]{AKLY2}. 
The aim of this paper is to generalise this example to obtain more finite dimensional algebras which do not satisfy DJHP. More precisely,  given two finite dimensional elementary algebras $A$ and $B$, we construct a $2\times 2$-matrix algebra $\Lambda(A,B)$ in Section~\ref{ss:property}. Here a finite-dimensional algebra is said to be \emph{elementary} if all its simple modules are 1-dimensional. Our main result is:
\begin{theorem*}[\ref{t:maintheorem}] DJHP fails for $\Lambda(A,B)$ in the following two cases:
\begin{itemize}
\item[(1)] both $A$ and $B$ are derived simple, but neither is simple;

\item[(2)] one of $A$ and $B$ is derived simple but nonsimple and the other one is of infinite global dimension.
\end{itemize} 
\end{theorem*}
As consequences of this result, we have
\begin{itemize}
\item[(a)] There are finite dimensional algebras of finite global dimension which do not satisfy DJHP. Indeed, if $A$ and $B$ are finite dimensional derived simple elementary algebras of positive finite global dimension, \emph{e.g.} Fibonacci algebras \cite{H,Membrillo-Hernandez94} and their generalisations \cite{LY16}, then $\Lambda(A,B)$ has finite global dimension and does not satisfy DJHP (Corollary~\ref{cor:finite-global-dimension}).

\item[(b)] Fix a local nonsimple finite dimensional elementary algebra $A$. Then for any finite dimensional elementary algebra $B$ applying the construction $\Lambda(?):=\Lambda(A,?)$ twice yields a finite dimensional algebra $\Lambda^2(B)$ which does not satisfy DJHP  (Corollary~\ref{cor:cardinality}).
\end{itemize}

During the preparation of this paper, Martin Kalck informed us that he found a family of finite dimensional algebras of global dimension 2 (hence quasi-hereditary) for which DJHP fails (see \cite[Proposition 3.4]{K}); Changchang Xi informed us that \cite[Theorem 1.1]{CX} can be used to construct finite dimensional algebras which do not satisfy DJHP. For two finite dimensional elementary algebras with the same number of isomorphism classes of simple modules, Chen and Xi constructed in the end of \cite[Section 5]{CX} an upper triangular $2\times 2$-matrix algebra, for which results similar to Theorem~\ref{t:maintheorem}, (a) and (b) can be obtained.

The paper is organised as follows: In Section \ref{s:preliminaries} we collect necessary definitions on recollements, stratifications and derived Jordan--H\"older property for later use. Section \ref{s:main result} contains our main result. We give the construction of the matrix algebra $\Lambda(A,B)$ and study its derived Jordan--H\"older property using stratifying ideals. In Section \ref{s:quiver-relation} we provide the underlying quiver and relations of $\Lambda(A,B)$.

Throughout the paper $k$ is a field, all algebras are finite dimensional algebras over $k$,  and all tensor products are over $k$ unless otherwise stated. For an algebra $A$, we denote by $\mathcal{D}(A)$ the unbounded derived category of right $A$-modules. For a quiver $Q$, we denote by $Q_0$ and $Q_1$ the sets of vertices and the set of arrows respectively. For an arrow $\alpha\in Q_1$, we denote by $\mathrm{s}(\alpha)$ and $\mathrm{t}(\alpha)$ the source and target of $\alpha$ respectively.

\medskip
\noindent{\it Acknowledgement.} The first-named author acknowledges support by National Natural Science Foundation of China No. 11671207, and by RFDP of Higher Education of China 20133207120013. 
The second-named author acknowledges support by National Natural Science Foundation of China No. 11671207.


\section{Preliminaries} \label{s:preliminaries}

In this section we recall the definition of recollements, stratifying ideals and stratifying recollements, stratifications, derived simplicity and derived Jordan--H\"older property of algebras.

\subsection{Recollements}
Let $A,B,C$ be algebras.
A \emph{recollement}~\cite{BBD} of
$\cd(A)$ by $\cd(B)$ and $\cd(C)$ is a diagram
\begin{eqnarray*}\xymatrix@!=6pc{\mathcal{D}(B)
\ar[r]|{i_*=i_!} &\mathcal{D}(A) \ar@<+2.5ex>[l]^{i^!}
\ar@<-2.5ex>[l]_{i^*} \ar[r]|{j^!=j^*} &
\mathcal{D}(C)\ar@<+2.5ex>[l]^{j_*}
\ar@<-2.5ex>[l]_{j_!}}
\end{eqnarray*}
of
$k$-linear triangle functors such that
\begin{enumerate}
\item $(i^\ast,i_\ast)$,\,$(i_!,i^!)$,\,$(j_!,j^!)$ ,\,$(j^\ast,j_\ast)$
are adjoint pairs;

\item  $i_\ast,\,j_\ast,\,j_!$  are full embeddings;

\item  $i^!\circ j_\ast=0$ (and thus also $j^!\circ i_!=0$ and
$i^\ast\circ j_!=0$);

\item  for each $M\in \mathcal{D}(A)$ there are triangles
\[
\xymatrix@R=0.3pc{
i_! i^!(M)\ar[r]& M\ar[r]& j_\ast j^\ast (M)\ar[r]& i_!i^!(M)[1]\\
j_! j^! (M)\ar[r]& M\ar[r]& i_\ast i^\ast(M)\ar[r]& j_!j^!(M)[1]
}
\]
where the first two maps in each triangle are given by adjunctions.
\end{enumerate}

We say the recollement is {\it trivial} if one of the algebras $B$ and $C$ is trivial, or equivalently, one
of the full embeddings $i_*$, $j_!$ and $j_*$ is a triangle
equivalence.

For an algebra $A$, we denote by $\rank(A)$ the number of nonisomorphic simple $A$-modules and call it the \emph{rank} of $A$.

\begin{proposition}\label{prop:invariants}
 Let $A,B,C$ be algebras and assume that there is a recollement of $\cd(A)$ by $\cd(B)$ and $\cd(C)$. Then
 \begin{itemize}
 \item[(i)] {\rm (\cite[Proposition 2.14]{AKLY2})}
 $A$ is of finite global dimension if and only if so are $B$ and $C$;
 \item[(ii)] {\rm (\cite[Proposition 6.5]{AKLY2})}
 $\rank(A)=\rank(B)+\rank(C)$.
 \end{itemize}
\end{proposition}

\subsection{Stratifying ideals and stratifying recollements}
\label{ss:stratifying-recollements}

Let $A$ be an algebra and $e\in A$ an idempotent. Following \cite[Definition 2.1.1]{CPS} we call the idempotent ideal $AeA$ a {\it stratifying ideal} of $A$ if the natural morphism $Ae\lten_{eAe}eA\to AeA$ is an isomorphism in $\cd(A)$. If $AeA$ is stratifying, then there exists a recollement of $\cd (A)$ by $\cd (A/AeA)$ and $\cd(eAe)$ (see for example \cite[Example 4.5]{AKL1}) 
\[ \label{d:recollement}
\xy (-46,0)*{\mathcal{D}(A/AeA)}; {\ar (-19,3)*{}; (-35,3)*{}_{i^\ast}};
{\ar (-35,0)*{}; (-19,0)*{}|{i_\ast=i_!}}; {\ar(-19,-3)*{}; (-35,-3)*{}^{i^!}};
(-12,0)*{{\mathcal D}(A)}; {\ar (11,3)*{}; (-5,3)*{}_{j_!}};
{\ar (-5,0)*{}; (11,0)*{}|{j^*=j^!}}; {\ar (11,-3)*{}; (-5,-3)^{j_*}};
(20,0)*{{\mathcal{D}(eAe)}};
\endxy
\]
where the six functors are derived functors given by
\[
\begin{array}{c}
i^*=-\lten_A A/AeA,\ \ i^!=\rhom_A(A/AeA,-),\\
i_*=\rhom_{A/AeA}(A/AeA,-)=-\lten_{A/AeA}A/AeA=i_!,\\
j_!=-\lten_{eAe} eA,\ \  j_*=\rhom_{eAe}(Ae,-), \\
j^!=\rhom_A(eA,-)=-\lten_A Ae=j^*.
\end{array}
\]
We refer to it as the {\it stratifying recollement} of $\mathcal{D}(A)$ induced by $e$. The following lemma is a left-module version of {\cite[Lemma 2.12]{AKLY2}}.

\begin{lemma}[] 
Let $A$ be an algebra and $e\in A$ an idempotent. Assume that $A/AeA$, as a left $A$-module, admits a projective resolution with components in $\add(Ae)$ except in degree $0$. Then $AeA$ is a stratifying ideal of $A$. \label{l:criterion for stratifying ideal}
\end{lemma}

\subsection{Derived simplicity of algebras}

An algebra $A$ is said to be {\em derived simple} if there is no nontrivial recollement of $\cd(A)$ by derived categories of algebras.
Known examples include  local algebras \cite[Proposition 4.11]{AKL2}, Fibonacci algebras  \cite{H,Membrillo-Hernandez94} and their generalisations \cite{LY16}, blocks of group algebras of finite groups and indecomposable symmetric algebras of
finite representation type \cite{LY12}, and indecomposable self-injective algebras of rank 2 with nonpositive Cartan determinant \cite[Remark 5]{HanQin16}. 

\subsection{Stratifications and the derived Jordan--H\"older property}

Roughly speaking, a
\emph{stratification} is a way of breaking up a given derived
category into simple pieces using recollements. More rigorously, for an algebra $A$ a
stratification of $\cd(A)$ is a full rooted binary tree whose root is the given
derived category $\cd(A)$, whose nodes are derived categories of algebras and whose
leaves are derived categories of derived simple algebras such that a node together with its two child
nodes form a nontrivial recollement unless it is a leaf.  The leaves are called the \emph{simple factors} of the
stratification. By abuse of language, we will also call the corresponding derived simple algebras the simple factors of the stratification. If we allow trivial recollements in the above definition, we call the binary tree a \emph{pseudostratification}. Given a pseudostratification, we obtain a stratification by deleting the subtrees with trivial nodes and contracting the edges which are equivalences.

\begin{lemma}\label{lem:splitting-stratification}
Let $A$ be an algebra with a decomposition $A=A_1\times\cdots\times A_s$ into the direct product of algebras $A_1,\ldots,A_s$. Then any stratification of $\cd(A)$ is the direct sum of pseudostratifications of $\cd(A_1),\ldots,\cd(A_s)$.
\end{lemma}
\begin{proof}
Let $B$ and $C$ be algebras such that there is a recollement of $\cd(A)$ by $\cd(B)$ and $\cd(C)$. It follows from \cite[Lemma 2.3]{AKLY2} that there are algebras $B_1,\ldots,B_s$, $C_1,\ldots,C_s$ (some of which may be trivial) such that $B=B_1\times\cdots\times B_s$ and $C=C_1\times\cdots\times C_s$ and there are recollements of $\cd(A_i)$ by $\cd(B_i)$ and $\cd(C_i)$, $i=1,\ldots,s$, which sum up to the given recollement. The desired result follows by induction.
\end{proof}

Analogous to the classical Jordan--H\"older property of finite groups, the derived Jordan--H\"older property of an algebra is about the existence and uniqueness of  finite stratifications. Let $A$ be an algebra. According to \cite[Corollary 7.4]{AKLY2}, any stratification of the derived category $\cd(A)$ is finite. If the simple factors in any stratification of $\cd(A)$ are unique up to ordering and up to derived equivalence, we say that the algebra $A$ satisfies the \emph{derived Jordan--H\"older property} (DJHP for short).
As examples, piecewise hereditary algebras \cite{AKL2,AKL3} and derived discrete algebras \cite{Q} satisfy DJHP.

\begin{proposition}\label{prop:DJHP-for-2-vertex-alg-with-finite-global-dimension}
Assume that $k$ is algebraically closed. If $A$ is of rank $2$ and of finite global dimension, then $A$ satisfies DJHP.
\end{proposition}
\begin{proof}
Assume that $A$ is not derived simple. Then there are algebras $B$ and $C$ such that there is a nontrivial recollement of $\cd(A)$ by $\cd(B)$ and $\cd(C)$. By Proposition~\ref{prop:invariants}, both $B$ and $C$ are of rank $1$ and of finite global dimension, so they are derived equivalent to $k$. Therefore, any stratification of $\cd(A)$ is a recollement and has simple factors two copies of $k$.
\end{proof}

According to \cite[Theorem 3.7]{LY16}, an algebra over an algebraically closed field which is of rank $2$ and of finite global dimension is either derived simple or derived equivalent to a quasi-hereditary algebra.




\section{Main result} \label{s:main result}

In this section we construct a matrix algebra $\Lambda = \Lambda(A,B)$ from two elementary algebras $A$ and $B$. We show that it has two natural stratifying ideals. By analysing the factors in the induced stratifying recollements, we give two cases when DJHP fails for $\Lambda$ (see Theorem \ref{t:maintheorem}).

\medskip

Recall that an algebra is said to be \emph{elementary} if all its simple modules are 1-dimensional, equivalently, if it is isomorphic to the quotient of the path algebra of a finite quiver by an admissible ideal. For an elementary algebra $A$ let $\rad(A)$ be the radical of the algebra $A$ and $\bar{A}=A/\rad(A)$ the semisimple quotient, which is the direct product of finitely many copies of $k$. We fix a retraction $\bar{A} \hookrightarrow A$ such that $\bar{A}$ becomes a subalgebra of $A$ (if $A$ is given by a quiver with relations, we take $\bar{A}$ to be the subspace spanned by all trivial paths). Then $A$ is isomorphic to $\bar{A} \oplus \rad(A)$ as a $k$-vector space. For $a\in A$, we will write $a=a_1+a_2$ with $a_1\in \bar{A}$ and $a_2\in \rad(A)$.

\subsection{The trivially twisted tensor product}\label{ss:tensor-product}

Let $A$ and $B$ be two elementary algebras.
Define a $k$-linear map $\tau: B \otimes A \ra A \otimes B$ as the composition
$$B \otimes A \twoheadrightarrow (B \otimes A)/ (\rad(B) \otimes \rad(A)) \xrightarrow{\text{flip}} (A \otimes B) / (\rad(A) \otimes \rad(B)) \hookrightarrow A \otimes B$$
where the $\text{flip}$ is defined by switching the two components in the tensor products,
$(A \otimes B) / (\rad(A) \otimes \rad (B))$ is identified with $\bar{A} \otimes \bar{B} \oplus \bar{A} \otimes \rad(B) \oplus \rad(A) \otimes \bar{B}$ and naturally embedded into $A\ten B$. 
Thus for $a\in A$ and $b\in B$ we have
\begin{align}\label{formula:tau}
\tau(b \otimes a) = a \otimes b_1 + a_1 \otimes b_2 = a_1 \otimes b_1 + a_2 \otimes b_1 + a_1 \otimes b_2=a\ten b-a_2\ten b_2.
\end{align}
In particular, $\tau(b \otimes a)=0$ if $a\in\rad(A)$ and $b\in\rad(B)$ and $\tau(b\otimes a)=a\otimes b$ if $a\in \bar{A}$ or $b\in \bar{B}$.

Denote by $\mu_A$ and $\mu_B$ the multiplications of the algebras $A$ and $B$, respectively. Define a $k$-linear map $\mu_\tau$ on the tensor product $A\otimes B$:
$$\mu_\tau: (A \otimes B) \otimes (A \otimes B) \xrightarrow{\id_A \otimes \tau \otimes \id_B} A \otimes A \otimes B \otimes B \xrightarrow{\mu_A \otimes \mu_B} A \otimes B.$$
More precisely, for $a,a' \in A$ and $b,b' \in B$, 
we have
\begin{align*}
\mu_\tau((a \otimes b) \otimes (a' \otimes b')) & =  \mu_A \otimes \mu_B (a \otimes (a' \otimes b_1 + a'_1 \otimes b_2) \otimes b')\\
& =  aa' \otimes b_1b' + aa'_1 \otimes b_2 b'\\
& = aa'\otimes bb'-aa'_2\otimes b_2 b'.
\end{align*}

\begin{proposition}\label{prop:twisted-tensor-product}
The vector space $A \otimes B$ with multiplication $\mu_\tau$ is an algebra.
\end{proposition}

\begin{proof}
By \cite[Proposition/Definition 2.3 and Remark 2.4 (1)]{CSV}, it is enough to prove
\begin{align*}\tau \circ (\id_B \otimes \mu_A) &= (\mu_A \otimes \id_B) \circ (\id_A \otimes \tau) \circ (\tau \otimes \id_A),\\
\tau \circ (\mu_B \otimes \id_A) &= (\id_A \otimes \mu_B) \circ (\tau \otimes \id_B) \circ (\id_B \otimes \tau).
\end{align*}
We only prove the first equality and the second is similar. 
Let $b\in B$ and $a,a'\in A$. Because $(aa')_1=a_1a'_1$, we have
\begin{align*}
\tau \circ (\id_B \otimes \mu_A)(b \otimes a \otimes a') & =  \tau( b \otimes aa')\\
& \stackrel{\eqref{formula:tau}}{=}  aa'\otimes b_1 + (aa')_1\otimes b_2\\
& =  aa' \otimes b_1 + a_1a'_1 \otimes b_2.
\end{align*}
On the other hand
\begin{align*}
(\mu_A \otimes \id_B) \circ (\id_A \otimes \tau) \circ (\tau \otimes \id_A)(b \otimes a \otimes a') & \stackrel{\eqref{formula:tau}}{=}
       (\mu_A \otimes \id_B) \circ (\id_A \otimes \tau)((a \otimes b_1 + a_1 \otimes b_2) \otimes a')\\
& =  (\mu_A \otimes \id_B)(a \otimes a' \otimes b_1 + a_1 \otimes a'_1 \otimes b_2)\\
& =  aa' \otimes b_1 + a_1a'_1 \otimes b_2.\qedhere
\end{align*}
\end{proof}

Following \cite{CX} we call the algebra in Proposition~\ref{prop:twisted-tensor-product} the {\it trivially twisted tensor product} of $A$ and $B$ (over $k$) and denote it by $A\ten_0 B$. If $A$ or $B$ is semisimple, then $A\ten_0 B$ is the usual tensor product $A\ten B$ of algebras. Moreover, $\bar{A}\ten B$ and $A\ten \bar{B}$ are subalgebras of $A\ten_0 B$.

\begin{lemma}\label{l:twisted tensor} 
\begin{itemize}
\item[(i)] $A\otimes \rad(B)$ and $\rad(A) \otimes B$ are two-sided ideals of $A \otimes_0 B$ with  quotients $A \otimes \bar{B}$ and $\bar{A} \otimes B$, respectively.

\item[(ii)] The restriction of $\mu_\tau$ to $(A \otimes \rad(B)) \otimes (\rad(A) \otimes B)$ vanishes.
\end{itemize}
\end{lemma}

\begin{proof}

(i) We check that $A \otimes \rad(B)$ is closed when multiplying with an element in $A \otimes_0 B$. The proof for $\rad(A) \otimes B$ is similar. For $a,a'\in A$ and $b=b_2 \in \rad(B),b' \in B$ we have
\begin{align*}
\mu_\tau((a \otimes b_2) \otimes (a' \otimes b')) & = \mu_A \otimes \mu_B (a \otimes  a'_1 \otimes b_2 \otimes b')\\
& = aa'_1 \otimes b_2 b' \in A \otimes \rad(B),
\end{align*}
and
\begin{align*}
\mu_\tau((a' \otimes b') \otimes (a \otimes b_2)) & \stackrel{\eqref{formula:tau}}{=}  \mu_A \otimes \mu_B (a' \otimes  (a \otimes b'_1 + a_1 \otimes b'_2)  \otimes b_2)\\
& =  a'a \otimes b'_1 b_2 + a'a_1 \otimes b'_2 b_2 \in A \otimes \rad(B).
\end{align*}

(ii) This follows from the fact that the restriction of $\tau$ to $\rad(B) \otimes \rad(A)$ vanishes.
\end{proof}

\subsection{The matrix algebra: construction and properties}\label{ss:construction} \label{ss:property}

Let $A$ and $B$ be two elementary algebras. We have the following data:
\begin{itemize}
\item[(1)] an $(\bar{A} \otimes B)$-$(A \otimes_0 B)$-bimodule $\bar{A} \otimes B$: the left $(\bar{A}\ten B)$-module structure is the free module of rank $1$; the
right $(A \otimes_0 B)$-module structure is induced from the quotient map $A \otimes_0 B \twoheadrightarrow \bar{A} \otimes B = (A\otimes_0 B)/ (\rad(A) \otimes B)$ in Lemma~\ref{l:twisted tensor} (i); explicitly, we have
\[
(a'\ten b')(a\ten b)(a''\ten b'')=a'aa''_1\ten b'bb''
\]
for $a\ten b, a'\ten b'\in \bar{A}\ten B$ and $a''\ten b''\in A\ten_0 B$;

\item[(2)] an $(A \otimes_0 B)$-$(\bar{A} \otimes B)$-bimodule $A \otimes \rad(B)$: $A \otimes \rad(B)$ as a two-sided ideal of $A \otimes_0 B$ is naturally an $(A\otimes_0 B)$-$(A \otimes_0 B)$-bimodule; by Lemma \ref{l:twisted tensor} (ii), $\rad(A) \otimes B$ annihilates $A \otimes \rad(B)$ from the right hand side, hence the right $(A\ten_0 B)$-module structure on $A \otimes \rad(B)$ factors through $\bar{A} \otimes B$ via the quotient map $A \otimes_0 B \twoheadrightarrow \bar{A} \otimes B = (A\otimes_0 B)/ (\rad(A) \otimes B)$ in Lemma~\ref{l:twisted tensor} (i); the right ${\bar A} \otimes B$-module structure on $A\ten \rad(B)$ can also be obtained via the embedding $\bar{A} \otimes B\hookrightarrow A\ten_0 B$ because the composition of this embedding with the quotient map $A \otimes_0 B \twoheadrightarrow \bar{A} \otimes B$ is the identity of $\bar{A}\ten B$; explicitly, we have
\[
(a'\ten b')(a\ten b)(a''\ten b'')=a'aa''\ten b'_1bb''+a'a_1a''\ten b'_2bb''
\]
for $a\ten b\in A\ten \rad(B)$, $a'\ten b'\in A\ten_0 B$ and $a''\ten b''\in\bar{A}\ten B$;

\item[(3)] an  $(A \otimes_0 B)$-$(A \otimes_0 B)$-bimodule homomorphism $$\varphi:(A \otimes \rad(B)) \otimes_{\bar{A} \otimes B} (\bar{A} \otimes B) \rightarrow A \otimes_0 B,~~(a\ten b)\ten (a'\ten b')\mapsto aa'\ten bb'$$ for $a\ten b\in A\ten \rad(B)$ and $a'\ten b'\in\bar{A}\ten B$; it is the composition of the canonical isomorphism $(A \otimes \rad(B)) \otimes_{\bar{A} \otimes B} (\bar{A} \otimes B)\stackrel{\cong}{\rightarrow} A\ten \rad(B)$ and the embedding $A\ten \rad(B)\hookrightarrow A\ten_0 B$;

\item[(4)] an $(\bar{A} {\otimes} B)$-$(\bar{A} {\otimes} B)$-bimodule homomorphism
\begin{align*}
\psi:(\bar{A} \otimes B) \otimes_{{A} \otimes_0 B} ({A} \otimes \rad(B)) \rightarrow \bar{A} {\otimes } B,~~(a\ten b)\ten (a'\ten b')\mapsto aa'_1\ten bb'
\end{align*}
for $a\ten b\in \bar{A}\ten B$ and $a'\ten b'\in A\ten \rad(B)$; it is the composition of the embedding $(\bar{A} \otimes B) \otimes_{{A} \otimes_0 B} ({A} \otimes \rad(B)) \hookrightarrow (\bar{A} \otimes B) \otimes_{{A} \otimes_0 B} ({A} \otimes_0 B)$ with the canonical isomorphism $(\bar{A} \otimes B) \otimes_{{A} \otimes_0 B} ({A} \otimes_0 B)\stackrel{\cong}{\rightarrow} \bar{A}\ten B$. \end{itemize}
Moreover, the following two diagrams are commutative
\begin{align*}
\xymatrix@C=1pc{
(\bar{A} \otimes B)\ten_{A\ten_0 B}(A \otimes \rad(B)) \otimes_{\bar{A} \otimes B} (\bar{A} \otimes B) \ar[rr]^(0.6){\id_{\bar{A}\ten B}\ten\varphi} \ar[d]^{\psi\ten\id_{\bar{A}\ten B}} & &(\bar{A} \otimes B)\ten_{A\ten_0 B} (A\ten_0 B)\ar[d]^\cong\\
(\bar{A}\ten B)\ten_{\bar{A}\ten B}(\bar{A}\ten B)\ar[rr]^(0.55)\cong && \bar{A}\ten B
}
\\
\xymatrix@C=1pc{
(A\ten\rad(B))\ten_{\bar{A}\ten B}(\bar{A}\ten B)\ten_{A\ten_0 B}(A\ten\rad(B))\ar[rr]^(0.6){\id_{A\ten\rad(B)}\ten\psi} \ar[d]^{\varphi\ten\id_{A\ten \rad(B)}} && (A\ten\rad(B))\ten_{\bar{A}\ten B} (\bar{A}\ten B)\ar[d]^\cong\\
(A\ten_0 B)\ten_{A\ten_0 B}(A\ten \rad(B))\ar[rr]^(0.6)\cong && A\ten\rad(B).
}
\end{align*}
We check the commutativity of the first diagram, and the commutativity of all other diagrams in the section can be checked in a similar way. For $a\ten b\in\bar{A}\ten B$, $a'\ten b'\in A\ten\rad(B)$ and $a''\ten b''\in \bar{A}\ten B$, we have
\begin{align*}
(a\ten b)\varphi((a'\ten b')\ten (a''\ten b''))&=(a\ten b)(a'a''\ten b'b'')=aa'_1a''\ten bb'b'',\\
\psi((a\ten b)\ten(a'\ten b'))(a''\ten b'')&=(aa'_1\ten bb')(a''\ten b'')=aa'_1a''\ten bb'b''.
\end{align*}
Therefore we obtain an algebra
$$\Lambda(A,B) := \begin{pmatrix} A \otimes_0 B & A \otimes \rad(B) \\ \bar{A} \otimes B & \bar{A} \otimes B \end{pmatrix},$$
whose addition is componentwise and whose multiplication is given by matrix multiplication together with the homomorphisms $\varphi$ and $\psi$. See for example \cite{Green82} for the study on matrix rings. In Section~\ref{s:quiver-relation} we will describe the quiver with relations of $\Lambda(A,B)$ in terms of those of $A$ and $B$.

\smallskip
Put $\Lambda=\Lambda(A,B)$. By abuse of notation we write $1$ for the unit elements in all of the algebras $A$, $B$ and $\Lambda$. Consider the two natural idempotents in $\Lambda$
$${\footnotesize e_1=\begin{pmatrix} 1\ten 1 & 0 \\ 0 & 0 \end{pmatrix},\hskip 20pt e_2 =\begin{pmatrix} 0 & 0 \\ 0 & 1\ten 1 \end{pmatrix}.}$$
Then in $\Lambda$ it holds that $1= e_1 + e_2$ and $e_1e_2 = 0 = e_2e_1$. Since $A$ is elementary, the rank of $A$ equals $\mathrm{dim}\bar{A}$. The rank of a general algebra is smaller than or equal to the cardinality of a complete set of primitive orthogonal idempotents of $A$; the equality holds if it is elementary.

\begin{lemma} The following hold for $\Lambda$: \label{l:idempotents in Lambda}
\begin{itemize}
\item[(i)] $e_1 \Lambda e_1 = A \otimes_0 B$ and  $\Lambda / \Lambda e_1 \Lambda \cong \bar{A} \otimes \bar{B}$;

\item[(ii)]  $e_2 \Lambda e_2 = \bar{A} \otimes B$ and $\Lambda / \Lambda e_2 \Lambda \cong {A} \otimes \bar{B}$;


\item[(iii)] $\Lambda$ is elementary with $\rank(\Lambda) = 2\rank(A) \rank(B)$.

\end{itemize}
\end{lemma}

\begin{proof}
(i) It is clear that $e_1 \Lambda e_1 = A \otimes_0 B$. The ideal generated by $e_1$ is
\begin{align*}
\Lambda e_1 \Lambda & =  \begin{pmatrix} A \otimes_0 B & A \otimes \rad(B) \\ \bar{A} \otimes B & \bar{A} \otimes B \end{pmatrix} \begin{pmatrix} 1\ten 1 & 0 \\ 0 & 0 \end{pmatrix}  \begin{pmatrix} A \otimes_0 B & A \otimes \rad(B) \\ \bar{A} \otimes B & \bar{A} \otimes B \end{pmatrix} \\
& =    \begin{pmatrix} A \otimes_0 B & 0 \\ \bar{A} \otimes B & 0 \end{pmatrix} \begin{pmatrix} A \otimes_0 B & A \otimes \rad(B) \\ \bar{A} \otimes B & \bar{A} \otimes B \end{pmatrix} \\
& =   \begin{pmatrix} A \otimes_0 B & A \otimes \rad(B) \\ \bar{A} \otimes B & \bar{A} \otimes \rad(B) \end{pmatrix}.
\end{align*}
Hence  $\Lambda / \Lambda e_1 \Lambda \cong \bar{A} \otimes B / (\bar{A} \otimes \rad(B)) \cong \bar{A} \otimes \bar{B}$.

(ii) Similar to (i).

(iii) Let $\{e_i\mid i\in I\}$ and $\{f_j\mid j\in J\}$ be complete sets of primitive orthogonal idempotents of $\bar{A}$ and $\bar{B}$, respectively. Then $e_i\ten f_j$ ($i\in I, j\in J$) is a primitive idempotent of $A\ten_0 B$ because $(e_i\ten f_j )(A\ten_0 B) (e_i\ten f_j)\stackrel{\eqref{formula:tau}}{=}e_iAe_i\ten_0 f_jBf_j$ is local. It follows that
\[
\{\begin{pmatrix}e_i\ten f_j & 0\\ 0 &0\end{pmatrix},\begin{pmatrix}0&0\\ 0 & e_i\ten f_j\end{pmatrix} \mid i\in I, j\in J\}
\]
is a complete set of primitive orthogonal idempotents of $\Lambda$. It follows that $\rank(\Lambda)\leq 2\rank(A)\rank(B)$.

On the other hand, the subspace
\begin{align*}
I&=\begin{pmatrix}
A\ten\rad(B)+\rad(A)\ten B & A\ten \rad(B)\\
\bar{A}\ten B & \bar{A}\ten \rad(B)
\end{pmatrix}
\end{align*}
is an ideal of $\Lambda$ with semisimple quotient algebra
\[
\Lambda/I=\begin{pmatrix}
\bar{A}\ten \bar{B} & 0\\
0 & \bar{A}\ten \bar{B}
\end{pmatrix}.
\]
Therefore $\rank(\Lambda)\geq 2\rank(A)\rank(B)$. Altogether we have $\rank(\Lambda)=2\rank(A)\rank(B)$ and $\Lambda$ is elementary. Moreover,
$I=\rad(\Lambda)$ and $\bar{\Lambda}=\Lambda/I$.
\end{proof}

Sometimes we can recover $B$ from $\Lambda$. When $A$ is local nonsimple, a complete set of primitive orthogonal idempotents of $\Lambda$ naturally fall into two groups, as shown in the proof below. Using one group we recover $B$. More precisely, take any complete set $e_1,\ldots,e_{2m}$ of primitive orthogonal idempotents of $\Lambda$. We may assume that $\Lambda/\Lambda(1-e_i)\Lambda$ is isomorphic to $A$ for $i=1,\ldots,m$ and to $k$ for $i=m+1,\ldots,2m$. Put $e=e_{m+1}+\ldots+e_{2m}$. Then $e\Lambda e\cong B$.
This idea does not work for general $A$. 

\begin{lemma}\label{lem:recovering-B}
Assume that $A$ is local nonsimple and let $B'$ be another elementary algebra. If $\Lambda=\Lambda(A,B)$ and $\Lambda(A,B')$ are isomorphic as algebras, then $B$ and $B'$ are isomorphic as algebras.
\end{lemma}

\begin{proof} Put $\Lambda'=\Lambda(A,B')$ and suppose that $\phi:\Lambda\rightarrow\Lambda'$ is an algebra isomorphism.
By Lemma~\ref{l:idempotents in Lambda} (iii), $B$ and $B'$ have the same rank. Let $\{f_1,\ldots,f_m\}$ and $\{f'_1,\ldots,f'_m\}$ be complete sets of primitive orthogonal idempotents of $\bar{B}$ and $\bar{B}'$, respectively. Then
\[
\{\begin{pmatrix}1\ten f_j & 0\\ 0 &0\end{pmatrix},\begin{pmatrix}0&0\\ 0 & 1\ten f_j\end{pmatrix} \mid j=1,\ldots,m\}
\]
and
\[
\{\begin{pmatrix}1\ten f'_j & 0\\ 0 &0\end{pmatrix},\begin{pmatrix}0&0\\ 0 & 1\ten f'_j\end{pmatrix} \mid j=1,\ldots,m\}
\]
are complete sets of primitive orthogonal idempotents of $\Lambda$ and $\Lambda'$, respectively. Moreover, for $1\leq j\leq m$,
the ideal of $\Lambda$ generated by  $1-\begin{pmatrix} 1\ten f_j & 0\\ 0 & 0\end{pmatrix}$ is
\begin{align*}
\Lambda (1-\begin{pmatrix} 1\ten f_j & 0\\ 0 & 0\end{pmatrix})\Lambda=\begin{pmatrix}
A\ten(\rad(B)+B(1-f_j)B) & A\ten \rad(B)\\
\bar{A}\ten B & \bar{A}\ten B
\end{pmatrix},
\end{align*}
so
\begin{align*}
\Lambda/\Lambda(1-\begin{pmatrix} 1\ten f_j & 0\\ 0 & 0\end{pmatrix})\Lambda&\cong A.
\end{align*}
Similarly, for $1\leq j\leq m$, we have
\begin{align*}
&\Lambda/\Lambda(1-\begin{pmatrix} 0 & 0\\ 0 & 1\ten f_j\end{pmatrix})\Lambda\cong k,\\
&\Lambda'/\Lambda'(1-\begin{pmatrix} 1\ten f'_{j} & 0\\ 0 & 0\end{pmatrix})\Lambda'\cong A,\\
&\Lambda'/\Lambda'(1-\begin{pmatrix} 0 & 0\\ 0 & 1\ten f'_{j'} \end{pmatrix})\Lambda'\cong k.
\end{align*}
Therefore, by the uniqueness of decomposition of the identity into primitive orthogonal idempotents (\cite[Theorem 3.4.1]{DrozdKirichenko}), there exist an invertible element $\lambda\in\Lambda'$ and a permutation $\sigma$ on $\{1,\ldots,m\}$ such that $\phi\begin{pmatrix} 0 & 0\\ 0 & 1\ten f_j\end{pmatrix}=\lambda\begin{pmatrix} 0 & 0\\ 0 & 1\ten f'_{\sigma(j)}\end{pmatrix}\lambda^{-1}$  for any $1\leq j\leq m$. So $\phi\begin{pmatrix}0 & 0\\ 0 & 1\ten 1\end{pmatrix}=\lambda\begin{pmatrix}0 & 0\\ 0 & 1\ten 1\end{pmatrix}\lambda^{-1}$, and we obtain a chain of algebra isomorphisms
\begin{align*}
B&= \begin{pmatrix}0 & 0\\ 0 & 1\ten 1\end{pmatrix}\Lambda\begin{pmatrix}0 & 0\\ 0 & 1\ten 1\end{pmatrix}\\
&\cong \phi\begin{pmatrix}0 & 0\\ 0 & 1\ten 1\end{pmatrix}\Lambda'\phi\begin{pmatrix}0 & 0\\ 0 & 1\ten 1\end{pmatrix}\\
&\cong \begin{pmatrix}0 & 0\\ 0 & 1\ten 1\end{pmatrix}\Lambda'\begin{pmatrix}0 & 0\\ 0 & 1\ten 1\end{pmatrix}\\
&= B'. \qedhere
\end{align*}
\end{proof}

\subsection{The matrix algebra: failure of DJHP}
Let $A$ and $B$ be two elementary algebras. In Section~\ref{ss:property} we defined a $2 \times 2$-matrix algebra $\Lambda = \Lambda(A,B)$ and two idempotents $e_1$ and $e_2$ of $\Lambda$.

\begin{proposition}\label{p:stratifying}
The idempotent ideals generated by $e_1$ and $e_2$ are both stratifying ideals of $\Lambda$. Consequently, $\cd(\Lambda)$ has two stratifying recollements, induced by $e_1$ and $e_2$ and denoted by $(R1)$ and $(R2)$ respectively, such that
\begin{itemize}
\item[--]
$(R1)$ is a recollement of $\cd(\Lambda)$ by $\cd(\Lambda / \Lambda e_1 \Lambda)$ and $\cd(e_1\Lambda e_1)$, where $\Lambda / \Lambda e_1 \Lambda\cong \bar{A} \otimes \bar{B}$, which is the direct product of $\rank(A)\rank(B)$ copies of $k$, and $e_1\Lambda e_1 = A\otimes_0 B$;
\item[--] $(R2)$ is a recollement of $\cd(\Lambda)$ by $\cd(\Lambda / \Lambda e_2 \Lambda)$ and $\cd(e_2 \Lambda e_2)$, where $\Lambda / \Lambda e_2 \Lambda \cong A \otimes \bar{B}$, which is the direct product of $\rank(B)$ copies of $A$, and $e_2 \Lambda e_2 = \bar{A} \otimes B$, which is the direct product of $\rank(A)$ copies of $B$.
\end{itemize}
\end{proposition}

We need some preparation for the proof of Proposition~\ref{p:stratifying}.
Let $M$ be a left $B$-module with $B$-action given by $\rho_M:B\ten M\ra M$. By \cite[Sections 3.6, 3.7 and 3.8]{CSV}, the two equalities in the proof of Proposition~\ref{prop:twisted-tensor-product} shows that the map
\[
\xymatrix@R=0.5pc@C=1pc{
(A \otimes B) \otimes (A \otimes M)  \ar[rrr]^{\id_A\otimes \tau \otimes \id_M} &&& A \otimes A \otimes B\otimes M \ar[r]^(0.6){\mu_A\otimes \rho_M}  & A\otimes M\\
(a\ten b)\ten (a'\ten m)\ar@{|->}[rrrr] &&&& aa'\ten b_1m+aa'_1\ten b_2m
}
\]
defines a left $(A\ten_0 B)$-module structure on $A\ten M$. We denote it by $A\ten_0 M$.  Note that the ideal $A\ten \rad(B)$ of $A\ten_0 B$ as a left $(A\otimes_0 B)$-module is $A\ten_0\rad(B)$.
For a left $B$-module homomorphism $f:M\to N$, the map
\[
\id_A\ten f: A\ten_0 M\longrightarrow A\ten_0 N,~~a\ten m\mapsto a\ten f(m)
\]
is a left $A\ten_0 B$-module homomorphism. 


Next we construct a left $\Lambda$-module from $M$. For information on modules over matrix rings, we refer to \cite{Green82}. Now we have a left $(A\ten_0 B)$-module $A\ten_0 M$ (with the twisted tensor action given above) and a left $(\bar{A}\ten B)$-module $\bar{A}\ten M$ (with the usual tensor action). Moreover, we have
\begin{itemize}
\item[--] a left $A\ten_0 B$-module homomorphism
\[
\xi=\xi_M:(A\ten \rad(B))\ten_{\bar{A}\ten B}(\bar{A}\ten M)\to A\ten_0 M,~~(a\ten b)\ten(a'\ten m)\mapsto aa'\ten bm,
\]
for $a\ten b\in A\ten\rad(B)$ and $a'\ten b'\in\bar{A}\ten M$,
\item[--] a left $\bar{A}\ten B$-module homomorphism
\[
\eta=\eta_M:(\bar{A}\ten B)\ten_{A\ten_0 B}(A\ten_0 M)\to \bar{A}\ten M,~~(a\ten b)\ten (a'\ten m)\mapsto aa'_1\ten bm,
\]
for $a\ten b\in\bar{A}\ten B$ and $a'\ten m\in A\ten_0 M$.
\end{itemize}
They permit the following commutative diagrams
\begin{align*}
\xymatrix@C=1pc{
(\bar{A} \otimes B)\ten_{A\ten_0 B}(A \otimes \rad(B)) \otimes_{\bar{A} \otimes B} (\bar{A} \otimes M) \ar[rr]^(0.6){\id_{\bar{A}\ten B}\ten\xi} \ar[d]^{\psi\ten\id_{\bar{A}\ten M}} & &(\bar{A} \otimes B)\ten_{A\ten_0 B} (A\ten_0 M)\ar[d]^\eta\\
(\bar{A}\ten B)\ten_{\bar{A}\ten B}(\bar{A}\ten M)\ar[rr]^(0.55)\cong && \bar{A}\ten M
}
\\
\xymatrix@C=1pc{
(A\ten\rad(B))\ten_{\bar{A}\ten B}(\bar{A}\ten B)\ten_{A\ten_0 B}(A\ten_0 M)\ar[rr]^(0.6){\id_{A\ten\rad(B)}\ten\eta} \ar[d]^{\varphi\ten\id_{A\ten_0 M}} && (A\ten\rad(B))\ten_{\bar{A}\ten B} (\bar{A}\ten M)\ar[d]^\xi\\
(A\ten_0 B)\ten_{A\ten_0 B}(A\ten_0 M)\ar[rr]^(0.6)\cong && A\ten_0 M.
}
\end{align*}
So $\xi$ and $\eta$ define a left $\Lambda$-module structure on
\[
\begin{pmatrix}
A\ten_0 M\\ \bar{A}\ten M
\end{pmatrix}.
\]
We denote this module by $\widetilde{M}$. Note that $\widetilde{B}$ is naturally isomorphic to $\Lambda e_1$ as a left $\Lambda$-module.

Let $f:M\to N$ be a left $B$-module homomorphism. Then there are commutative diagrams
\begin{align*}
\xymatrix{
(\bar{A}\ten B)_{A\ten_0 B}(A\ten_0 M)\ar[r]^(0.7){\eta_M} \ar[d]^{\id_{\bar{A}\ten B}\ten (\id_A\ten f)}& \bar{A}\ten M\ar[d]^{\id_{\bar{A}}\ten f}\\
(\bar{A}\ten B)_{A\ten_0 B}(A\ten_0 N)\ar[r]^(0.7){\eta_N} & \bar{A}\ten N
}
\xymatrix{
(A\ten \rad(B))\ten_{\bar{A}\ten B} (\bar{A}\ten M)\ar[r]^(0.7){\xi_M}\ar[d]^{(\id_{A\ten \rad(B)})\ten (\id_{\bar{A}}\ten f)} & A\ten_0 M\ar[d]^{\id_A\ten f}\\
(A\ten\rad(B))\ten_{\bar{A}\ten B} (\bar{A}\ten N)\ar[r]^(0.7){\xi_N} & A\ten_0 N.
}
\end{align*}
It follows that
\[
\widetilde{f}:=\begin{pmatrix}
\id_A\ten f & \\
 & \id_{\bar{A}}\ten f
\end{pmatrix}
:\widetilde{M}
\longrightarrow
\widetilde{N}
\]
is a left $\Lambda$-module homomorphism.
Therefore, for a complex
\[
M^\bullet=\ldots\rightarrow M^{i-1}\stackrel{d^{i-1}}{\rightarrow} M^i\stackrel{d^i}{\rightarrow} M^{i+1}\rightarrow\ldots
\]
of left $B$-modules,
the sequence
\[
\widetilde{M^\bullet}:=\ldots\rightarrow
\widetilde{M^{i-1}}
\stackrel{\widetilde{d^{i-1}}}{\longrightarrow} \widetilde{M^{i}}
\stackrel{\widetilde{d^{i}}}{\longrightarrow} \widetilde{M^{i+1}}\rightarrow\ldots
\]
is a complex of left $\Lambda$-modules. Moreover, $M^\bullet$ is acyclic if and only if so is $\widetilde{M^\bullet}$ because as a complex of vector spaces $\widetilde{M^\bullet}$ is just $(A\oplus \bar{A})\ten M^\bullet$. Therefore we have

\begin{lemma}\label{lem:tensoring-projective-resolution-for-lambda}
Let $M$ be a left $B$-module and let $P^\bullet$ be a free resolution of $M$ over $B$. Then $\widetilde{P^\bullet}$ is a projective resolution of $\widetilde{M}$ over $\Lambda$ with all components in $\add(\Lambda e_1)$.
\end{lemma}

\begin{proof}[Proof of Proposition~\ref{p:stratifying}] In the second statement the existence of two stratifying recollements follows from the first statement since stratifying ideals induce stratifying recollements (Section~\ref{ss:stratifying-recollements}), and the forms of the algebras $\Lambda/\Lambda e_i\Lambda$ and $e_i\Lambda e_i$ ($i=1,2$) follow from Lemma~\ref{l:idempotents in Lambda}. To prove the first statement, we construct a projective resolution of the left $\Lambda$-module $\Lambda/ \Lambda e_i \Lambda$ ($i=1,2$) all of whose components but the one in degree $0$ belong to $\add(\Lambda e_i)$. The statement follows then by Lemma \ref{l:criterion for stratifying ideal}.

Multiplying $\begin{pmatrix} 0 & 0\\ 1\ten 1 & 0\end{pmatrix}$ on the right defines an injective $\Lambda$-module homomorphism
\[
\Lambda e_2\longrightarrow \Lambda e_1
\]
with quotient $\Lambda/\Lambda e_2\Lambda$. So there is a short exact sequence
\[
0\rightarrow \Lambda e_2\rightarrow \Lambda e_1\rightarrow \Lambda/\Lambda e_2\Lambda \rightarrow 0,
\]
which is a projective $\Lambda$-resolution of $\Lambda/\Lambda e_2\Lambda$ with the desired property.

Consider the projection from $\Lambda e_2$ to $\Lambda/\Lambda e_1\Lambda$. Its kernel is naturally isomorphic to
\[
\widetilde{\rad(B)}=\begin{pmatrix}A\ten_0 \rad(B) \\ \bar{A}\ten \rad(B)\end{pmatrix}.
\]
By Lemma~\ref{lem:tensoring-projective-resolution-for-lambda}, $\widetilde{\rad(B)}$ admits a projective $\Lambda$-resolution with all components in $\add(\Lambda e_1)$. This together with the projection from $\Lambda e_2$ to $\Lambda / \Lambda e_1 \Lambda$ provides a projective $\Lambda$-resolution of $\Lambda / \Lambda e_1 \Lambda$ with the desired property.
\end{proof}

\begin{corollary} The following statements are equivalent: \label{c:finite global dimension}
\begin{itemize}
\item[(i)] both $A$ and $B$ have finite global dimension;

\item[(ii)] the matrix algebra $\Lambda$ has finite global dimension;

\item[(iii)] the trivially twisted tensor product $A \otimes_0 B$ has finite global dimension.
\end{itemize}
\end{corollary}

\begin{proof} Applying Proposition~\ref{prop:invariants} (i) to the stratifying recollements $(R1)$ and $(R2)$ in Proposition~\ref{p:stratifying} we obtain the equivalence between (i) and (ii)  and between (i) and (iii), respectively.
\end{proof}

The following theorem is the main result of this paper.

\begin{theorem}  In the following two cases the matrix algebra $\Lambda$ does not satisfy the derived Jordan--H\"older property:
\begin{itemize}
\item[(1)] both $A$ and $B$ are derived simple but nonsimple;

\item[(2)] one of $A$ and $B$ is derived simple but nonsimple, and the other one has infinite global dimension.
\end{itemize} \label{t:maintheorem}
\end{theorem}


\begin{proof} For a stratification $\cs$ of $\cd(A)$ let $s(\cs)$ denote the multiplicity of $k$ as simple factors of $\cs$.

Consider the two stratifying recollements $(R1)$ and $(R2)$ in Proposition~\ref{p:stratifying}. Recall that $\Lambda/\Lambda e_1\Lambda\cong \bar{A}\ten \bar{B}$, $e_1\Lambda e_1= A\ten_0 B$, $\Lambda/\Lambda e_2\Lambda \cong A\ten \bar{B}$ and $e_2\Lambda e_2= \bar{A}\ten B$. By \cite[Lemma 7.1]{AKLY2}, $s(\cs_{\Lambda/\Lambda e_1\Lambda})=\rank(A)\rank(B)$ for any stratification $\cs_{\Lambda/\Lambda e_1\Lambda}$ of $\cd(\Lambda/\Lambda e_1\Lambda)$ and hence $s(\cs_1)\geq \rank(A)\rank(B)$ for any stratification $\cs_1$ of $\cd(A)$ refining $(R1)$.

We will show that in both cases (1) and (2), $s(\cs_2)<\rank(A)\rank(B)$ for any stratification $\cs_2$ refining the recollement $(R2)$. It follows that DJHP fails for $\Lambda$.
In both cases we may assume that $A$ is derived simple but nonsimple. Then it follows from \cite[Lemma 7.1]{AKLY2} that any stratification of $\cd(\Lambda/\Lambda e_2\Lambda)$ has simple factors $\rank(B)$ copies of $A$. Therefore it suffices to show that $s(\cs_{e_2\Lambda e_2})<\rank(A)\rank(B)$ for any stratification $\cs_{e_2\Lambda e_2}$ of $\cd(e_2\Lambda e_2)$. In view of Lemma~\ref{lem:splitting-stratification}, this is equivalent to showing that $s(\cs_{B})<\rank(B)$ for any stratification $\cs_{B}$ of $\cd(B)$. This is true in case (1) because $B$ is derived simple nonsimple. In case (2) this follows from Proposition~\ref{prop:invariants} since $B$ has infinite global dimension.
\end{proof}

\begin{remark}
The \emph{length} of a stratification is the number of its simple factors. In case (1) of Theorem~\ref{t:maintheorem}, if $A$ and $B$ are not local, namely, the ranks of $A$ and $B$ are greater than $1$, then the length of any stratification refining $(R1)$ is different from that of any stratification refining $(R2)$. Indeed, it follows from the proof of Theroem~\ref{t:maintheorem} that the length of any stratification refining $(R1)$ is greater than or equal to $1+\rank(A)\rank(B)$, while the length of any stratification refining $(R2)$ is precisely $\rank(A)+\rank(B)$.
\end{remark}

\begin{remark}
If $A$ (respectively, $B$) does not satisfy DJHP, nor does $\Lambda$. Indeed, we can start with the recollement $(R2)$ and obtain stratifications with different simple factors. Assume that both $A$ and $B$ satisfy DJHP. Following the proof of Theorem~\ref{t:maintheorem} we see that $\Lambda$ does not satisfy DJHP if $s(\cs_A)\rank(B)+s(\cs_B)\rank(A)<\rank(A)\rank(B)$ for a/any stratification $\cs_A$ of $\cd(A)$ and a/any stratification $\cs_B$ of $\cd(B)$. For example, this is the case if a/any stratification of $\cd(A)$ (respectively, $\cd(B)$) has no simple factors $k$, and $B$ (respectively, $A$) has infinite global dimension or a/any stratification of $\cd(B)$ (respectively, $\cd(A)$) has length strictly smaller than $\rank(B)$.
\end{remark}

Proposition~\ref{prop:DJHP-for-2-vertex-alg-with-finite-global-dimension} states that over an algebraically closed field all rank $2$ algebras of finite global dimension satisfy DJHP. This is far from being true for rank $2$ algebras of infinite global dimension.

\begin{corollary}\label{cor:local} Assume that both $A$ and $B$ are local nonsimple. Then $\Lambda$ does not satisfy DJHP.
\end{corollary}
\begin{proof}
This follows from Theorem~\ref{t:maintheorem} because local algebras are derived simple by \cite[Proposition 4.11]{AKL2}.
\end{proof}

We will see in Example~\ref{ex:akly2-example} that $\Lambda(k[x]/(x^2),k[x]/(x^2))$ is the algebra in \cite[Example 7.6]{AKLY2} which does not satisfy DJHP.

Examples of derived simple algebras of positive finite global dimension are given in \cite{H,Membrillo-Hernandez94,LY16}. So the next corollary shows that there are algebras of finite global dimension which do not satisfy DJHP.
In \cite[Proposition 3.4]{K} Kalck provides a family of finite dimensional algebras of global dimension 2 (hence quasi-hereditary) for which DJHP fails.

\begin{corollary}\label{cor:finite-global-dimension} Assume that both $A$ and $B$ are derived simple of positive finite global dimension. Then $\Lambda$ has finite global dimension and does not satisfy DJHP.
\end{corollary}
\begin{proof}
By Theorem~\ref{t:maintheorem}, it suffices to show that $\Lambda$ has finite global dimension,  which follows from Corollary \ref{c:finite global dimension}.
\end{proof}

The next corollary together with Lemma~\ref{lem:recovering-B} shows that the cardinality of the set of elementary algebras which do not satisfy DJHP is the same as the cardinality of the set of all elementary algebras.

\begin{corollary}\label{cor:cardinality} Assume that $A$ is local nonsimple. Then $\Lambda(A,\Lambda(A,B))$ does not satisfy DJHP.
\end{corollary}

\begin{proof}
By Theorem~\ref{t:maintheorem}, it suffices to show that $\Lambda(A,B)$ has infinite global dimension,  which follows from Corollary \ref{c:finite global dimension}.
\end{proof}


\section{Quiver and relations}\label{s:quiver-relation}

Let $A$ and $B$ be elementary algebras and let
$$\Lambda=\Lambda(A,B) := \begin{pmatrix} A \otimes_0 B & A \otimes \rad(B) \\ \bar{A} \otimes B & \bar{A} \otimes B \end{pmatrix},$$
be defined as in Section~\ref{ss:construction}. The aim of this section is to describe the quiver with relations of $\Lambda$ in terms of those of $A$ and $B$.

\begin{theorem}\label{thm:quiver-with-relations}
Let $A = kQ /(\mathcal{R})$ and $B = kP /(\mathcal{S})$, where $Q$ and $P$ are finite quivers, and $\mathcal{R}$ and $\cs$ are sets of minimal relations of $A$ and $B$, respectively. Then as an algebra $\Lambda$ is isomorphic to $k\Gamma/(\ci)$, where
\begin{itemize}
\item
the set $\Gamma_0$ of vertices  is $\{i^j,~j^i\mid i\in Q_0, j\in P_0\}$, namely, it is two copies of $Q_0\times P_0$;
\item the set $\Gamma_1$ consists of three types of arrows:
\begin{itemize}
\item $\alpha^j: \mathrm{s}(\alpha)^j \ra \mathrm{t}(\alpha)^j$, for $\alpha\in Q_1$ and $j\in P_0$,

\item $c(i,j): i^j \ra j^i$, for $i\in Q_0$ and $j\in P_0$,

\item $\beta^i: \mathrm{s}(\beta)^i \ra i^{\mathrm{t}(\beta)}$, for $\beta \in P_1$ and $i\in Q_0$;
\end{itemize}
\item
the set $\mathcal{I}$ consists of three types of relations:
\begin{itemize}
\item $r^j$, for $r\in\mathcal{R}$ and $j\in P_0$,

\item $c(i,j) \alpha^j$, for $i\in Q_0$, $j\in P_0$ and $\alpha \in Q_1$ with $\mathrm{t}(\alpha)=i$,

\item $s^i$, for $s\in \mathcal{S}$ and $i\in Q_0$,
\end{itemize}
where
\begin{align*}r^j& =\sum \lambda_{i_1,\ldots, i_s} \alpha_{i_1}^j \cdots \alpha_{i_s}^j
\end{align*}
for $r = \sum \lambda_{i_1, \ldots, i_s} \alpha_{i_1} \cdots \alpha_{i_s} \in \mathcal{R}$ and $j\in P_0$, and
\begin{align*}s^i &= \sum \mu_{j_1,\ldots,j_t} \beta_{j_1}^i c(i,\mathrm{s}(\beta_{j_1})) \beta_{j_2}^i c(i,\mathrm{s}(\beta_{j_2}))\cdots \beta_{j_{t-1}}^i c(i,\mathrm{s}(\beta_{j_{t-1}})) \beta_{j_t}^i\\
&=\sum \mu_{j_1,\ldots,j_t} \beta_{j_1}^i c \beta_{j_2}^i c\cdots \beta_{j_{t-1}}^i c \beta_{j_t}^i
\end{align*} for $s = \sum \mu_{j_1, \ldots, j_t} \beta_{j_1} \cdots \beta_{j_t} \in \mathcal{S}$ and $i \in Q_0$, where $c=\sum_{i\in Q_0,j\in P_0}c(i,j)$.
\end{itemize}
\end{theorem}

Before proving Theorem~\ref{thm:quiver-with-relations}, we first give two examples illustrating the construction of the quiver with relations in the theorem.

\begin{example}
Let $Q$ and $P$ be the following quivers with relations
\begin{align*}\xymatrix{1\ar[r]^{\alpha_1} & 2\ar[r]^{\alpha_2} & 3\ar[r]^{\alpha_3} & 4} ,&~~\alpha_3\alpha_2\alpha_1\\
\xymatrix{5\ar[r]^{\beta_1} & 6\ar[r]^{\beta_2} & 7 },&~~\beta_2\beta_1.
\end{align*}
Then $\Gamma$ is given by
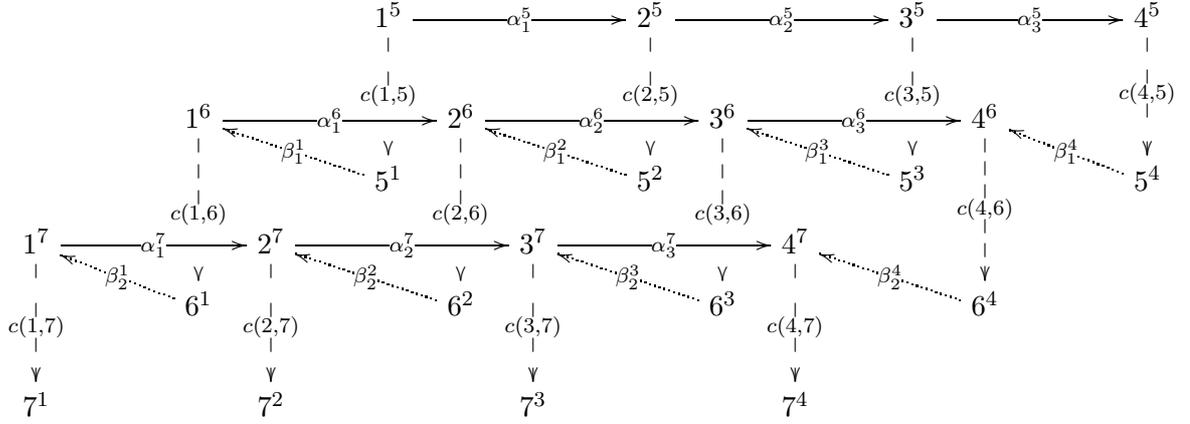
\[
\xymatrix@C=0.7pc@R=0.5pc{
  & & && & & {\color{green}1^5} \ar[rrrr]|{\alpha_1^5} \ar@[blue]@{-->}'[dd]|(0.75){\color{blue}c(1,5)}[ddd]  &  & & & {\color{green}2^5} \ar@[blue]@{-->}'[dd]|(0.75){\color{blue}c(2,5)}[ddd] \ar[rrrr]|{\alpha_2^5} &  & & & {\color{green}3^5}\ar@[blue]@{-->}'[dd]|(0.75){\color{blue}c(3,5)}[ddd] \ar[rrrr]|{\alpha_3^5} &  & && {\color{green}4^5} \ar@[blue]@{-->}[ddd]|(0.45){\color{blue}c(4,5)}      \\ \\
  & &&{\color{green}1^6} \ar[rrrr]|{\alpha_1^6}\ar@[blue]@{-->}'[dd]|(0.75){\color{blue}c(1,6)}[ddd]      & & &  & {\color{green}2^6} \ar@[blue]@{-->}'[dd]|(0.75){\color{blue}c(2,6)}[ddd] \ar[rrrr]|{\alpha_2^6} & & & & {\color{green}3^6} \ar@[blue]@{-->}'[dd]|(0.75){\color{blue}c(3,6)}[ddd] \ar[rrrr]|{\alpha_3^6} & & & & {\color{green}4^6} \ar@[blue]@{-->}[ddd]|(0.47){\color{blue}c(4,6)} \\
  &  &&& && {\color{red}5^1} \ar@[violet]@{..>}[ulll]|{\color{violet}\beta_1^1}      &  & & & {\color{red}5^2} \ar@[violet]@{..>}[ulll]|{\color{violet}\beta_1^2}   &  & &&  {\color{red}5^3} \ar@[violet]@{..>}[ulll]|{\color{violet}\beta_1^3}  & & & & {\color{red}5^4} \ar@[violet]@{..>}[ulll]|{\color{violet}\beta_1^4}           \\
  {\color{green}1^7} \ar[rrrr]|{\alpha_1^7}\ar@[blue]@{-->}[ddd]|{\color{blue}c(1,7)} &  & &&  {\color{green}2^7} \ar[rrrr]|{\alpha_2^7} \ar@[blue]@{-->}[ddd]|{\color{blue}c(2,7)} &  & & & {\color{green}3^7} \ar[rrrr]|{\alpha_3^7} \ar@[blue]@{-->}[ddd]|{\color{blue}c(3,7)} &  & && {\color{green}4^7} \ar@[blue]@{-->}[ddd]|{\color{blue}c(4,7)} \\
  & & & {\color{red}6^1} \ar@[violet]@{..>}[ulll]|{\color{violet}\beta_2^1}      &  & & & {\color{red}6^2} \ar@[violet]@{..>}[ulll]|{\color{violet}\beta_2^2}  &  & && {\color{red}6^3} \ar@[violet]@{..>}[ulll]|{\color{violet}\beta_2^3}  & & & & {\color{red}6^4} \ar@[violet]@{..>}[ulll]|{\color{violet}\beta_2^4}  \\  \\
  {\color{red}7^1} & &&& {\color{red}7^2} &&& & {\color{red}7^3} &&& & {\color{red}7^4} }
\]
and $\ci$ consists of
\begin{itemize}
\item[--] $\alpha_3^j\alpha_2^j\alpha_1^j$, $j=5,6,7$;

\item[--] $c(i,j)\alpha_{i-1}^j$, $i=2,3,4$, $j=5,6,7$;

\item[--] $\beta_2^i c(i,6) \beta_1^i$, $i=1,2,3,4$.
\end{itemize}

\end{example}

\begin{example}\label{ex:akly2-example}
Let $A=k[x]/(x^2)$ and $B=k[y]/(y^2)$. The quiver of $A$ (respectively, $B$) has a unique vertex, denoted by $a$ (respectively, $b$), and one loop $\alpha$ (respectively, $\beta$). By Theorem~\ref{thm:quiver-with-relations}, the quiver with relations of $\Lambda$ is
\[
\xymatrix@R=0.2pc{{\color{green}a^b}\ar@(ul,dl)_{\alpha^b}\ar@[blue]@{-->}@<.7ex>[ddd]^{\color{blue}c(a,b)} \\
&,~~\beta^a c(a,b) \beta^a,(\alpha^b)^2,c(a,b)\alpha^b.
\\
\\
{\color{red}b^a}\ar@[violet]@{..>}@<.7ex>[uuu]^{\color{violet}\beta^a}}
\]
As $A$ and $B$ are both local nonsimple, DJHP fails for $\Lambda$ by Corollary~\ref{cor:local}. Note that $\Lambda$ is of rank $2$, so a stratification of $\mathcal{D}(\Lambda)$ is just a recollement. The factors of $(R1)$ are $\bar{A}\ten\bar{B}= k$ and $A\ten_0 B\cong k\langle x,y\rangle/(x^2,y^2,yx)$, while the factors of $(R2)$ are  $A\ten\bar{B}= k[x]/x^2$ and $\bar{A}\ten B= k[y]/y^2$, as also shown in \cite[Example 7.6]{AKLY2}.
\end{example}

The rest of this section is devoted to the proof of Theorem~\ref{thm:quiver-with-relations}.

\begin{proof}[Proof of Theorem~\ref{thm:quiver-with-relations}]
Recall from the proof of Lemma~\ref{l:idempotents in Lambda} (iii) that
\begin{align*}
\rad(\Lambda)&=\begin{pmatrix}
A\ten\rad(B)+\rad(A)\ten B & A\ten \rad(B)\\
\bar{A}\ten B & \bar{A}\ten \rad(B)
\end{pmatrix}\\
& = \begin{pmatrix} (\rad A \otimes \bar{B}) \oplus (\bar{A} \otimes \rad B) \oplus (\rad A \otimes_0 \rad B) & A \otimes \rad B \\ \bar{A} \otimes B & \bar{A} \otimes \rad B \end{pmatrix}
\end{align*}
and
\[
\bar{\Lambda}=\begin{pmatrix}
\bar{A}\ten \bar{B} & 0\\
0 & \bar{A}\ten \bar{B}
\end{pmatrix}.
\]

We write $e_i$ ($i\in Q_0$) and $f_j$ ($j\in P_0$) for the primitive idempotents of $A$ and $B$ respectively, and $e_{i^j}, e_{j^i}$ the primitive idempotents in $k\Gamma$ corresponding to $i^j$ and $j^i$. Then the map
\[
e_{i^j}\mapsto \begin{pmatrix} e_i \otimes f_j & 0 \\ 0 & 0 \end{pmatrix} , \quad e_{j^i}\mapsto\begin{pmatrix} 0 & 0 \\ 0 & e_i \otimes f_j \end{pmatrix},\quad (i\in Q_0,~j\in P_0)
\]
defines an algebra isomorphism from $k\Gamma_0$ to $\bar{\Lambda}$. We identify these two semisimple algebras via this isomorphism.

The square of the radical of $\Lambda$ is
{\footnotesize
\begin{align*}
\rad^2 \Lambda & =  (\rad \Lambda)^2 \\
& =  \begin{pmatrix} \rad^2(A\otimes_0 B) \oplus (A \otimes \rad B) (\bar{A} \otimes B) & \rad(A\otimes_0 B) A \otimes \rad B + (A \otimes \rad B)(\bar{A} \otimes \rad B) \\ (\bar{A} \otimes B)\rad(A \otimes_0 B) + (\bar{A} \otimes \rad B)(\bar{A} \otimes B) & (\bar{A} \otimes B)(A \otimes \rad B) + \rad^2(\bar{A} \otimes B) \end{pmatrix}\\
& =  \begin{pmatrix} (\rad^2 A \otimes \bar{B}) \oplus (A \otimes \rad B) & (\rad A \otimes \rad B) \oplus (\bar{A} \otimes \rad^2 B) \\ \bar{A} \otimes \rad B  &  \bar{A} \otimes \rad B \end{pmatrix}.
\end{align*}}
Hence
\begin{align*}
\rad \Lambda / \rad^2 \Lambda  & =  \begin{pmatrix} (\rad A/ \rad^2 A) \otimes \bar{B} & \bar{A}  \otimes( \rad B / \rad^2 B) \\ \bar{A} \otimes \bar{B} & 0 \end{pmatrix}.
\end{align*}
It follows that the map
\[
\alpha^j \mapsto \begin{pmatrix} \alpha \otimes f_j & 0 \\ 0 & 0 \end{pmatrix}, \quad c(i,j)\mapsto\begin{pmatrix} 0 & 0 \\ e_i \otimes f_j & 0 \end{pmatrix}, \quad \beta^i \mapsto \begin{pmatrix} 0 & e_i \otimes \beta \\ 0 & 0\end{pmatrix} \quad (i\in Q_0,~j\in P_0)
\]
defines a $\bar{\Lambda}$-$\bar{\Lambda}$-bimodule isomorphism from the $k$-vector space spanned by the arrows of $\Gamma$ to $\rad(\Lambda)/\rad^2(\Lambda)$. This isomorphism then extends to a surjective algebra homomorphism
$$\tilde{\Phi}: k\Gamma \twoheadrightarrow \Lambda.$$

For any path $q =\alpha_1\alpha_2 \cdots \alpha_s$ of $Q$ and $j\in P_0$, $q^j:=\alpha_1^j \alpha_2^j \cdots \alpha_s^j$ is a path of $\Gamma$ and we have
$$\tilde\Phi(q^j) = \begin{pmatrix} q \otimes f_j & 0 \\ 0 & 0 \end{pmatrix}.$$
For any path $p =\beta_1 \beta_2 \cdots \beta_s$ of $P$ and $i \in Q_0$ we have
\begin{align*}
\tilde\Phi(c \beta_1^i c \beta_2^i c \cdots c \beta_s^i c ) & =  \begin{pmatrix} 0 & 0 \\ e_i \otimes \beta & 0  \end{pmatrix},\\
\tilde\Phi(c \beta_1^i c \beta_2^i c \cdots c \beta_s^i) & =  \begin{pmatrix} 0 & 0 \\ 0 & e_i \otimes \beta  \end{pmatrix},\\
\tilde\Phi(\beta_1^i c \beta_2^i c \cdots c \beta_s^i ) & =  \begin{pmatrix} 0 &  e_i \otimes \beta\\ 0 & 0  \end{pmatrix}.
\end{align*}
Therefore, the three types of relations in $\mathcal{I}$ belong to the kernel of $\tilde\Phi$: for $r\in\mathcal{R}$ and $j\in P_0$, $$\tilde{\Phi}(r^j) = \begin{pmatrix} r\otimes f_j & 0 \\ 0 & 0 \end{pmatrix} =0;$$
for $i\in Q_0$, $j\in P_0$ and $\alpha \in Q_1$ with $\mathrm{t}(a)=i$,
$$\tilde{\Phi}(c(i,j) \alpha^j) = \begin{pmatrix} 0 & 0 \\ e_i \otimes f_j & 0 \end{pmatrix} \begin{pmatrix} \alpha \otimes f_j & 0 \\ 0 & 0 \end{pmatrix} = \begin{pmatrix} 0 & 0 \\ \alpha \otimes f_j & 0 \end{pmatrix} =0;$$
for $s = \sum \mu_{j_1, \ldots, j_t} \beta_{j_1} \cdots \beta_{j_t} \in \mathcal{S}$ and $i \in Q_0$,
\begin{align*}
\tilde{\Phi}(s^i) & =  \tilde{\Phi}(\sum \mu_{j_1,\ldots,j_t} \beta_{j_1}^i c\beta_{j_2}^i c \cdots \beta_{j_{t-1}}^i c \beta_{j_t}^i)\\
& =   \sum \mu_{j_1,\ldots,j_t} \begin{pmatrix} 0 & e_i \otimes \beta_{j_1}\beta_{j_2} \cdots \beta_{j_t} \\ 0 & 0 \end{pmatrix} \\
& =  \begin{pmatrix} 0 & e_i\ten s \\ 0 & 0 \end{pmatrix} = 0.
\end{align*}
Hence $\tilde\Phi$ factors through a surjective algebra homomorphism
$$\Phi: \tilde{\Lambda} = k\Gamma /(\mathcal{I}) \twoheadrightarrow \Lambda.$$
We will show that $\Phi$ is bijective and thus an algebra isomorphism.

By abuse of notation we denote by $e_1$ the sum of all the idempotents $e_{i^j}$'s and $e_2$ the sum of all $e_{j^i}$'s in $\tilde\Lambda$. Then $\Phi(e_1) = e_1$ and $\Phi(e_2) = e_2$. Since $e_1$ and $e_2$ are orthogonal to each other, both algebras $\tilde\Lambda$ and $\Lambda$ decompose as $2\times 2$-matrix algebras and the decomposition is compatible with $\Phi$. We write $\Phi_{ij}$ for the induced map from $e_i \tilde\Lambda e_j$ to $e_i \Lambda e_j$ ($i,j = 1,2$). Then $\Phi = \bigoplus_{i,j=1}^{2} \Phi_{ij}$. The bijectivity of $\Phi$ is equivalent to the bijectivity of all $\Phi_{ij}$. 

Bijectivity of $\Phi_{11}: e_1 \tilde\Lambda e_1 \ra e_1 \Lambda e_1 = A \otimes_0 B$: The trivially twisted tensor product $A\otimes _0 B$ as  a $k$-algebra is generated by $e_i\otimes f_j$, $\alpha\otimes f_j$ and $e_i\otimes \beta$ for $i\in Q_0$, $j\in P_0$, $\alpha\in Q_1$ and $\beta\in P_1$, with generating relations $r\otimes f_j$ for $r\in\mathcal{R}$, $e_i\otimes s$ for $s\in \mathcal{S}$ and $(e_i \otimes \beta)(\alpha \otimes f_j)$. The preimages of $\alpha \otimes f_j$ and $e_i\otimes \beta$ in $e_1\tilde\Lambda e_1$ are $\alpha^j$ and $\beta^i c$ respectively. They also satisfy the generating relations of $A \otimes_0 B$. Indeed for $r=\sum \lambda_{i_1,\ldots, i_s} \alpha_{i_1} \cdots \alpha_{i_s} \in \mathcal{R}$, $r^j = \sum \lambda_{i_1,\ldots, i_s} \alpha_{i_1}^j \cdots \alpha_{i_s}^j=0$. For $s = \sum \mu_{j_1, \ldots, j_t} \beta_{j_1} \cdots \beta_{j_t} \in \mathcal{S}$, $ \sum \mu_{j_1,\ldots,j_t} \beta_{j_1}^i c \beta_{j_2}^i c \cdots \beta_{j_t}^i c =s^i c = 0$. For $i \in Q_0$, $j\in P_0$, $\alpha \in Q_1$ and $\beta \in P_1$, we have $\beta^i c \alpha^j =0$. Note that $e_1 \tilde\Lambda e_1$ as a $k$-algebra is generated by $e_{i^j}$, $\alpha^j$ and $\beta^ic$. Therefore $\Phi_{11}$ must be bijective.

Bijectivity of $\Phi_{22}: e_2 \tilde\Lambda e_2 \ra e_2 \Lambda e_2 = \bar{A} \otimes B$. The tensor product $\bar{A} \otimes B$ as a $k$-algebra is generated by $e_i \otimes f_j$ and $e_i \otimes \beta$ for $i\in Q_0$ and $\beta \in P_1$ with generating relations $e_i \otimes s$ for $s\in \mathcal{S}$. The preimages of $e_i \otimes f_j$ and $e_i \otimes \beta$ in $e_2\tilde\Lambda e_2$ are $e_{j^i}$ and $c\beta^i$ respectively. They satisfy the generating relations of $\bar{A} \otimes B$. Indeed for $s = \sum \mu_{j_1, \ldots, j_t} \beta_{j_1} \cdots \beta_{j_t} \in \mathcal{S}$, $c s^i = \mu_{j_1,\ldots,j_t} c \beta_{j_1}^i c \beta_{j_2}^i c \cdots \beta_{j_t}^i = 0$. An a $k$-algebra $e_2\tilde\Lambda e_2$ is generated by $e_{j^i}$ and $c\beta^i$, so $\Phi_{12}$ must be bijective.

Bijectivity of $\Phi_{21}: e_2 \tilde\Lambda e_1 \ra e_2 \Lambda e_1 = \bar{A} \otimes B$. For $i \in Q_0$, $j\in P_0$ and $\beta \in P_1$, the preimages of $e_i \otimes f_j$ and $e_i \otimes \beta$ are $c(i,j)$ and $c\beta^i c$ respectively. Consider the right multiplication map $r_c$ from $e_2\tilde\Lambda  e_2$ to $e_2\tilde\Lambda e_1$, sending an element $x$ to $xc$. This map is surjective according to the shape of $\Gamma$ and the description of $\ci$. Hence $\dim_k e_2 \tilde\Lambda e_1 \leq \dim_k e_2 \tilde\Lambda e_2 = \dim_k e_2 \Lambda e_2  = \dim_k e_2 \Lambda e_1$. But the surjectivity of $\Phi_{21}$ implies $\dim_k e_2 \tilde\Lambda e_1 \geq \dim_k e_2  \Lambda e_1$. Hence $e_2 \tilde\Lambda e_1$ and  $e_2 \Lambda e_1$ have the same dimension and $\Phi_{21}$ is bijective.

Finally for the bijectivity of $\Phi_{12}: e_1 \tilde\Lambda e_2 \ra e_1 \Lambda e_2 = A \otimes \rad B$, consider for each $c(i,j)$ the right multiplication map $r_{c(i,j)}: \tilde\Lambda e_{j^i} \ra \tilde\Lambda e_{i^j}$, and the sum $r_c = \sum r_{c(i,j)}: \tilde\Lambda e_2 \ra \tilde\Lambda e_1$. By Lemma \ref{l:injectivity} below, each $r_{c(i,j)}$ is injective, so is $r_c$. Let $r_c^1: e_1 \tilde\Lambda e_2 \ra e_1 \tilde\Lambda e_1$ be the corresponding summand of $r_c$. It follows that $r_c^1$ is injective. The following commutative diagram
\[
\xymatrix@C=0.7pc{
  e_1\tilde\Lambda e_2 \ar[d]_{\Phi_{12}} \ar[rrrrrr]^{r_c^1} &&&&&& e_1\tilde\Lambda e_1 \ar[d]^{\Phi_{11}} \\
  e_1 \Lambda e_2 \ar@{=}[r] & A \otimes \rad B \ar[rrrr]^{\mathrm{inc}} &&&& A \otimes_0 B\ar@{=}[r] & e_1\Lambda e_1 }
\]
shows that $\Phi_{12}$ is injective, and hence bijective.
\end{proof}

\begin{lemma} Let $A=kQ/(\mathcal{R})$ where $Q$ is a finite quiver and $\mathcal{R}$ is a set of minimal relations. Suppose $c: \mathrm{s}(c) \ra \mathrm{t}(c)$ is a path in $Q$ such that any path appearing in $\mathcal{R}$ is not of the form $pc$ for some path $p$ in $Q$. Then the right multiplication map $r_c: A e_{\mathrm{t}(c)} \ra A e_{\mathrm{s}(c)},~~x\mapsto xc$, is injective. \label{l:injectivity}
\end{lemma}

\begin{proof}
Put $\tilde{A}=kQ$ and $I=(\mathcal{R})$. We have the following commutative diagram
\[\xymatrix{
  0 \ar[r] & I e_{\mathrm{t}(c)} \ar[d]_{\alpha} \ar[r] & \tilde{A} e_{\mathrm{t}(c)} \ar[d]_{\beta} \ar[r] &  A e_{\mathrm{t}(c)} \ar[d]_{r_c} \ar[r] & 0 \\
  0 \ar[r] & I e_{\mathrm{s}(c)} \ar[r] & \tilde{A} e_{\mathrm{s}(c)} \ar[r] & A e_{\mathrm{s}(c)} \ar[r] & 0   }\]
where the rows are exact and all the vertical maps are right multiplication with $c$. It is clear that $\alpha$ and $\beta$ are injective. Hence $r_c$ is injective if and only if $\Im(\beta) \cap I e_{\mathrm{s}(c)} = \Im(\alpha)$, \ie $\tilde{A} c \cap I e_{\mathrm{s}(c)} = Ic$ in $\tilde {A}$. It suffices to show `$\subseteq$', \ie if $x\in\tilde{A}$ satisfies $xc \in I e_{\mathrm{s}(c)}$, then $x\in I$. Write $xc = \sum_{r\in\mathcal{R}} a_1 r a_2$ with $a_1, a_2 \in \tilde{A}$, and write further $a_2 = a_2' + a_2''$ such that all paths appearing in $a_2'$ are of the form $pc$  for some path $p$ in $\tilde{A}$, and none of the paths appearing in $a_2''$ is of such form. Then $xc = \sum_{r\in \mathcal{R}} a_1 r a_2' + \sum_{r\in\mathcal{R}} a_1 r a_2''$ in $\tilde{A}$. The first summand is of the form $x'c$ for some $x' \in I$ and the second summand must be zero. Hence $xc = x'c$, namely $\beta(x) = \beta(x')$. It follows from the injectivity of $\beta$ that $x=x'\in I$.
\end{proof}


\end{document}